\documentclass[10pt,a4paper,reqno,cite]{amsproc}
\usepackage{cite}
\usepackage{mathrsfs}
\newtheorem{theorem}{Theorem}[section]
\newtheorem{lemma}[theorem]{Lemma}
\newtheorem{corollary}[theorem]{Corollary}
\newtheorem{problem}[theorem]{Problem}

\theoremstyle{definition}

\theoremstyle{remark}

\theoremstyle{proposition}

\numberwithin{equation}{section}

\usepackage{amssymb}
\usepackage[mathscr]{eucal}
\begin{document}
\title[Mehler-Fock transform of arbitrary order]{The Mehler-Fock transform of arbitrary order}
\author [S. Ranjan]{Sudhanshu Ranjan$^1$}
\author [A. Prasad]{Akhilesh Prasad$^2$}

\thanks{{\textbf{}} \\	
$^{1,2}$Department of  Mathematics and Computing, Indian Institute of Technology (Indian School of Mines), Dhanbad-826004, India,
E-mail: sudhanshum$1997@$gmail.com$^{1}$, apr$\_$bhu$@$yahoo.com$^2$\\	
{\bf Corresponding Author:} Akhilesh Prasad$^2$, E-mail: apr$\_$bhu$@$yahoo.com	\\
\\
{\bf Acknowledgement:} This work was supported by National Board of Higher Mathematics, Govt. of India under sanction no. 02011/7/2021 NBHM (R.P)/{RD}-II/8164}

\date{}


\maketitle
\begin{abstract}
In this paper, we construct the Mehler–Fock transform of arbitrary order $\mathfrak{MF}_{\mu,k}$, by removing the restriction $\mathfrak{R}\mu>{-\frac{1}{2}}$ with the help of a positive integer $k$ such that $\mathfrak{R}(\mu+k)>{-\frac{1}{2}}$ and obtain some transformation formulae based on the differential operator $M_{\mu}N_{\mu}$. Furthermore, we define the Mehler–Fock translation and convolution operators of arbitrary order and demonstrate that the factorization property for the convolution holds in the test function space $ \mathbb{H}_\mu$. 
\end{abstract}

\section{Introduction }
The Mehler-Fock transform (MFT) has long been recognized as an effective tool in signal processing and probability theory, as well as for solving boundary value problems in mathematical physics and elasticity theory. It belongs to the class of index integral transforms, which are characterized by the dependence on one of the indices or parameters of the kernel function. Examples include the Kontorovich–Lebedev transform, the Whittaker transforms, the Lebedev-Skalskaya transform, and so on (see \cite {index, Maan,Gupt}). The MFT was initially introduced by Mehler \cite{mehler}, and its inversion formula was established by Fock \cite{fock}. Further in-depth investigations of the zero-order MFT have been studied by Gonz\'alez and Negr\'in\cite {Bn1}, Sneddon \cite{Sn}, Hayek and Gonz\'alez\cite {Ha2},  Yakubovich \cite {Yaku2}, Glaeske and Hess \cite{per}, Prasad et al. \cite{PS} etc. Additionally, the generalization of the MFT to the order $\mu$ has also been examined in detail by various researchers \cite{Hm,low,AP1,gonz3,Glaeske,Ha1,gonz1} so far.
The MFT of order $\mu\in \mathbb{C}$ with $\mathfrak{R}\mu>{-\frac{1}{2}}$ of a complex valued function $f$ is defined as \cite{gonz1}:
\begin{eqnarray}
(\mathfrak{MF}_\mu f)(\tau)=\int_1^\infty f(x) P_{-\frac{1}{2}+i\tau}^{-\mu}(x)dx,\quad \tau\in \mathbb{R}_+=(0,\infty), \label{1.1}
\end{eqnarray} 
and it can be inverted by 
\begin{eqnarray}
	f(x)=(\mathfrak{MF}^{-1}_\mu \mathfrak{MF}_\mu f)(x)= \frac{1}{\pi}\int_0^\infty\tau \sinh(\pi\tau)\Psi(\tau, \mu)P_{-\frac{1}{2}+i\tau}^{-\mu}(x)(\mathfrak{MF}_\mu f)(\tau)d\tau, \label{1.2}
\end{eqnarray}
where $x\in\mathbb{I}=(1,\infty)$, $\Psi(\tau,\mu)=\Gamma(\mu+\frac{1}{2}+i\tau)\Gamma(\mu+\frac{1}{2}-i\tau)$ and
$P_{-\frac{1}{2}+i\tau}^{-\mu}(x)$ is the associated Legendre function of first kind.\\
 In this paper, we introduce another method of extending the $\mu{th}$ order MFT to arbitrary negative real values of the order $\mu$.
\\

\textbf{Plan of the paper:} Section 2 presents recurrence relations for the associated Legendre function of the first kind, certain differential and integral operators, and the test function space $\mathbb{H}_\mu(\mathbb{I})$, which are used in the subsequent analysis. Section 3 is concerned with the MFT of arbitrary order, its inverse, and some of their properties. Section 4 is devoted to some operational calculus based on the differential operator $M_\mu N_\mu$. Section 5 deals with the Mehler-Fock convolution of arbitrary order and demonstrates the factorization property of this convolution in the space $ \mathbb{H}_\mu(\mathbb{I})$. Finally, Section 6 discusses applications of the arbitrary order MFT.
\section{Preliminary Results and Notations }
In this section, we present some preliminary results, notations, and known properties related to the $\mu$th order Mehler-Fock transform that will be used throughout the paper.
\\

As per \cite[ p. 259]{Ha1}, the product of functions  $P_{-\frac{1}{2}+i\tau}^{-\mu}$  states
\begin{eqnarray}
	P_{-\frac{1}{2}+i\tau}^{-\mu}(x)P_{-\frac{1}{2}+i\tau}^{-\mu}(y)=\int_1^\infty K_\mu(x,y,z)P_{-\frac{1}{2}+i\tau}^{-\mu}(z) dz,\label{2.1}
\end{eqnarray}
with a kernel that is symmetric in $x, y, z$,
\begin{eqnarray*}
	K_\mu(x,y,z) = \begin{cases}
		\frac{2^{\mu}\Gamma(\mu+1)}{\pi\Gamma(2\mu+1)}Y(x,y,z)^{\mu-1/2}[(x^2-1)(y^2-1)(z^2-1)]^{-\frac{\mu}{2}}, &\quad \text{for}\quad Y>0
		, \\ 
		0 &\quad \text{for}\quad Y\leq0,
	\end{cases}
\end{eqnarray*}
where 
\begin{eqnarray*}
Y(x,y,z)=2xyz+1-x^2-y^2-z^2. 
\end{eqnarray*}
Moreover, from \eqref{2.1} and \eqref{1.2}, $K_\mu(x,y,z)$ can be expressed as\\
\begin{eqnarray}
	K_\mu(x,y,z)=\frac{1}{\pi}\int_0^\infty\tau \sinh(\pi\tau)\Psi(\tau,\mu)P_{-\frac{1}{2}+i\tau}^{-\mu}(x)P_{-\frac{1}{2}+i\tau}^{-\mu}(y)P_{-\frac{1}{2}+i\tau}^{-\mu}(z)d\tau,\label{2.2}
\end{eqnarray}
where $\Psi(\tau,\mu)$ as above.\\
According to \cite{bry,AP2}, the recurrence relations of the associated Legendre function of first kind are as follows:
	\begin{eqnarray}
		(i)~  D_x^n [(x^2-1)^{-\frac{\mu}{2}}P_{-\frac{1}{2}+i\tau}^{-\mu}(x)]&=&(-1)^{\frac{n}{2}}(\mu+1/2+i\tau)_{n}(\mu+1/2-i\tau)_{n}\nonumber \\
		&\times&(x^2-1)^{-\frac{(\mu+n)}{2}}P_{-\frac{1}{2}+i\tau}^{-(\mu+n)}(x), \label{2.3}\\
		(ii) ~  D_x^n [(x^2-1)^{\frac{\mu}{2}}P_{-\frac{1}{2}+i\tau}^{-\mu}(x)]&=&(-1)^{\frac{n}{2}} (x^2-1)^{\frac{\mu-n}{2}}P_{-\frac{1}{2}+i\tau}^{-(\mu-n)}(x). \label{2.4}
	\end{eqnarray}
	where, $(\alpha)_n=\frac{\Gamma(\alpha+n)}{\Gamma(\alpha)},~ n \in {\mathbb{N}}$ and $ ~(\alpha)_0=1$.\\
	From \cite{AP2}, the Dirac delta function for MFT is given by 
	\begin{eqnarray*}
		\delta(x-y)=\frac{1}{\pi}\int\limits_0^\infty\tau \sinh(\pi\tau)\Gamma(\mu+\frac{1}{2}+i\tau)\Gamma(\mu+\frac{1}{2}-i\tau)P_{-\frac{1}{2}+i\tau}^{-\mu}(x)P_{-\frac{1}{2}+i\tau}^{-\mu}(y)d\tau.
	\end{eqnarray*}	
The test function space $ \mathbb{H}_\mu(\mathbb{I})$ of all complex-valued infinitely differentiable functions $\phi$ on $\mathbb{I}$ such that 
\begin{eqnarray}
	\mathfrak{Y}_{m,k}^\mu(\phi)=\sup\limits_{x\in\mathbb{I}}\bigg|(x^2-1)^m D_x^k (x^2-1)^{-\frac{\mu}{2}}\phi(x)\bigg|< \infty, \label{2.5}
\end{eqnarray}
 for all non-negative integers $m$ and $k$. $ \mathbb{H}_\mu(\mathbb{I})$ is equipped with the topology generated by the family of seminorms ${\mathfrak{Y}_{m,k}^\mu}$. $ \mathbb{H}_\mu(\mathbb{I})$ is sequentially complete and therefore a Frech\'et space. $\mathbb{H}_\mu^{\prime}(\mathbb{I})$ is the dual of $ \mathbb{H}_\mu(\mathbb{I})$. The members of $\mathbb{H}_\mu^{\prime}(\mathbb{I})$ are of slow growth as $x\to \infty$.
 We define the generalized Mehler-Fock transform $\mathfrak{MF}_\mu^{\prime}$ on $\mathbb{H}_\mu^{\prime}(\mathbb{I})$ by extending the Parseval relation as \cite{low}. In particular, if $~\mathfrak{R}\mu>{-\frac{1}{2}}$, $\Phi\in \mathbb{H}_\mu(\mathbb{I})$, $\phi=\mathfrak{MF}^{-1}_\mu\Phi$, and $f\in\mathbb{H}_\mu^{\prime}(\mathbb{I})$, then we define $F=\mathfrak{MF}_\mu^{\prime}f$ as a functional on $\mathbb{H}_\mu(\mathbb{I})$ by 
\begin{eqnarray}
	\langle F,\Phi \rangle	_{(\mathbb{R}_+;\frac{1}{\pi}\tau \sinh(\pi\tau)\Psi(\tau,\mu)d\tau)}=\langle f,\phi \rangle_{(\mathbb{I}; dx)}.\label{2.6}
\end{eqnarray}
Thus $\mathfrak{MF}_\mu^{\prime}$ is the adjoint of $\mathfrak{MF}^{-1}_\mu$ with respect to the weighted measure $(\mathbb{R}_+;\frac{1}{\pi}\tau\\ \sinh(\pi\tau)\Psi(\tau,\mu)d\tau)$.\\
The differential operators $M_\mu$, $N_\mu$ and the integral operators $M_\mu^{-1}$, $N_\mu^{-1}$ are defined in \cite{AP2} as
\begin{eqnarray}
 M_\mu\phi(x) &=& (x^2-1)^{-\frac{\mu}{2}}D_x(x^2-1)^{\frac{\mu+1}{2}}\phi(x) \label{2.7},
	\\
 N_\mu\phi(x) &=&(x^2-1)^{\frac{\mu+1}{2}}D_x(x^2-1)^{-\frac{\mu}{2}}\phi(x)\label{2.8},
	\\
 N_\mu^{-1}\phi(x) &=&(x^2-1)^{\frac{\mu}{2}}\int\limits_\infty^x (t^2-1)^{-\frac{(\mu+1)}{2}}\phi(t)dt \label{2.9},
	\\
 M_\mu^{-1}\phi(x) &=& (x^2-1)^{-\frac{(\mu+1)}{2}}\int\limits_\infty^x (t^2-1)^{\frac{\mu}{2}}\phi(t)dt.\label{2.10}
\end{eqnarray}
The differential operator $A_\mu$ is defined as
\begin{eqnarray}
	A_\mu = M_\mu N_\mu &=&(x^2-1)^{-\frac{\mu}{2}}D_x(x^2-1)^{\mu+1}D_x(x^2-1)^{-\frac{\mu}{2}} \nonumber\\
	&=&(x^2-1)D^2_x+2xD_x-\bigg[\mu(\mu+1)+{\frac{\mu^2}{x^2-1}}\bigg].\label{2.11}
\end{eqnarray}

\section{The Mehler-Fock transform of arbitrary order }
   The $\mu{th}$ order MFT \eqref{1.1} is defined with the restriction $\mathfrak{R}\mu>{-\frac{1}{2}}$. In this work, we have eliminated this restriction with the help of a positive integer $k$ such that $\mathfrak{R}(\mu+k)>{-\frac{1}{2}}$ and define the extended Mehler-Fock transform $\mathfrak{MF}_{\mu,k}$ of any function $\phi\in \mathbb{H}_\mu(\mathbb{I})$ as 
\begin{eqnarray}
\Phi(\tau):=(\mathfrak{MF}_{\mu,k} \phi)(\tau)=(-1)^{\frac{k}{2}}\big(\mathfrak{MF}_{\mu+k}N_{\mu+k-1}\cdots N_\mu \phi\big)(\tau), \label{3.1}
\end{eqnarray} 
and the inverse of extended Mehler-Fock transform $\mathfrak{MF}_{\mu,k}^{-1}$ of any $\Phi\in \mathbb{H}_\mu(\mathbb{I})$ is defined by
\begin{eqnarray}
	\phi(x):=(\mathfrak{MF}_{\mu,k}^{-1} \Phi)(x)=(-1)^{\frac{k}{2}}\big(N_\mu^{-1}\cdots N_{\mu+k-1}^{-1}\mathfrak{MF}_{\mu+k}^{-1} \Phi\big)(x), \label{3.2}
\end{eqnarray} 
where, $\mathfrak{MF}_{\mu+k}$ is defined as \eqref{1.1}, and $\mathfrak{MF}_{\mu+k}^{-1}$ is defined as \eqref{1.2}, wherein $\mu$ is replaced by $\mu+k$.\\
\begin{theorem}
	Let $\mu$ be a complex number and $k$ a positive integer such that $\mathfrak{R}(\mu+k)>{-\frac{1}{2}}$. On $\mathbb{H}_\mu(\mathbb{I})$, the extended Mehler-Fock transform $\mathfrak{MF}_{\mu,k}$ coincides with $\mathfrak{MF}_{\mu}$ defined by \eqref{1.1} whenever $\mathfrak{R}(\mu)>{-\frac{1}{2}}$.
	\begin{proof} 
		Suppose that $\phi \in\mathbb{H}_\mu(\mathbb{I})$, then consider the case for $k=1$:
		\begin{eqnarray*}
		(\mathfrak{MF}_{\mu,1} \phi)(\tau)&=&(-1)^{\frac{1}{2}}\big(\mathfrak{MF}_{\mu+1}N_{\mu}\phi\big)(\tau)\\
		&=&(-1)^{\frac{1}{2}}\int_1^\infty(x^2-1)^\frac{\mu+1}{2}P_{-\frac{1}{2}+i\tau}^{-(\mu+1)}(x)[D_x(x^2-1)^{-\frac{\mu}{2}}\phi(x)]dx.
		\end{eqnarray*}
		Integrating by parts and using the recurrence relation \eqref{2.4}, we get
		\begin{eqnarray*}
			(\mathfrak{MF}_{\mu,1} \phi)(\tau)
			&=&\int_1^\infty P_{-\frac{1}{2}+i\tau}^{-\mu}(x)\phi(x)dx=(\mathfrak{MF}_{\mu} \phi)(\tau).
		\end{eqnarray*}
	Therefore, $$(\mathfrak{MF}_{\mu}\phi)(\tau)=(-1)^{\frac{1}{2}}\big(\mathfrak{MF}_{\mu+1}N_{\mu}\phi\big)(\tau).$$
	Now for $k=2$:
		\begin{eqnarray*}
		(\mathfrak{MF}_{\mu,2} \phi)(\tau)&=&(-1)^{\frac{2}{2}}\big(\mathfrak{MF}_{\mu+2}N_{\mu+1}N_{\mu}\phi\big)(\tau)\\
		&=&(-1)\int_1^\infty(x^2-1)^\frac{\mu+2}{2}P_{-\frac{1}{2}+i\tau}^{-(\mu+2)}(x)[D_x^2(x^2-1)^{-\frac{\mu}{2}}\phi(x)]dx.
	\end{eqnarray*}
Integrating by parts two times and using the recurrence relation \eqref{2.4}, we get
\begin{eqnarray*}
	(\mathfrak{MF}_{\mu,2} \phi)(\tau)
	&=&\int_1^\infty P_{-\frac{1}{2}+i\tau}^{-\mu}(x)\phi(x)dx=(\mathfrak{MF}_{\mu} \phi)(\tau)
\end{eqnarray*}
Therefore, $$(\mathfrak{MF}_{\mu}\phi)(\tau)=(-1)^{\frac{2}{2}}\big(\mathfrak{MF}_{\mu+2}N_{\mu+1}N_{\mu}\phi\big)(\tau).$$
Similarly, by induction on $k$, we have 
\begin{eqnarray*}
	(\mathfrak{MF}_{\mu}\phi)(\tau)
	&=&(-1)^{\frac{k}{2}}\big(\mathfrak{MF}_{\mu+k}N_{\mu+k-1}\cdots N_\mu \phi\big)(\tau)\\
	&=&(\mathfrak{MF}_{\mu,k} \phi)(\tau).
	\end{eqnarray*}
	\end{proof}
\end{theorem}

\begin{corollary}
		For any positive integers $k$ and $p$ such that $k>p>-\mathfrak{R}(\mu)-\frac{1}{2}$, $\mathfrak{MF}_{\mu,k}=\mathfrak{MF}_{\mu,p}$ on $\mathbb{H}_\mu(\mathbb{I})$.
	\end{corollary}
\begin{proof} According to the definition of $\mathfrak{MF}_{\mu,k}$ given in \eqref{3.1}, $\mathfrak{MF}_{\mu,k}$ is independent of the choice of $k$ provided that $\mathfrak{R}(\mu+k)>{-\frac{1}{2}}$. Indeed if $k>p>-\mathfrak{R}(\mu)-\frac{1}{2}$, then by Theorem 3.1, we have $\mathfrak{MF}_{\mu+p,k-p}=\mathfrak{MF}_{\mu+p}$, hence
	\begin{eqnarray*}
		&&(\mathfrak{MF}_{\mu,k} \phi)(\tau)\\
		&&=(-1)^{\frac{k}{2}}\big(\mathfrak{MF}_{\mu+k}N_{\mu+k-1}\cdots N_\mu \phi\big)(\tau)\\
		&&=(-1)^{\frac{k}{2}}\big(\mathfrak{MF}_{\mu+k}N_{\mu+k-1}\cdots N_{\mu+p}N_{\mu+p-1}\cdots N_\mu \phi\big)(\tau)\\
		&&=(-1)^{\frac{p}{2}}(-1)^{\frac{k-p}{2}}\big(\mathfrak{MF}_{\mu+p+k-p}N_{\mu+p+k-p-1}\cdots N_{\mu+p}N_{\mu+p-1}\cdots N_\mu \phi\big)(\tau)\\
		&&=(-1)^{\frac{p}{2}}\big(\mathfrak{MF}_{\mu+p,k-p}N_{\mu+p-1}\cdots N_\mu \phi\big)(\tau)\\
        &&=(-1)^{\frac{p}{2}}\big(\mathfrak{MF}_{\mu+p}N_{\mu+p-1}\cdots N_\mu \phi\big)(\tau)=(\mathfrak{MF}_{\mu,p} \phi)(\tau).
	\end{eqnarray*}
	\end{proof}
\begin{lemma}
	For $\Phi\in\mathbb{H}_\mu(\mathbb{I})$, we have
	\begin{eqnarray*}
		N_\mu^{-1}\cdots N_{\mu+k-2}^{-1}N_{\mu+k-1}^{-1}\big(\mathfrak{MF}_{\mu+k}^{-1} \Phi\big)(x)=(-1)^{-\frac{k}{2}}\big(\mathfrak{MF}_{\mu}^{-1} \Phi\big)(x).
	\end{eqnarray*}
\end{lemma}
\begin{proof}
 Using \eqref{2.9} and \eqref{1.2}, we have
 	\begin{eqnarray*}
 	&&N_{\mu+k-1}^{-1}\big(\mathfrak{MF}_{\mu+k}^{-1} \Phi\big)(x)\\
 	&&=(x^2-1)^{\frac{\mu+k-1}{2}}\int_\infty^x (t^2-1)^{-\frac{(\mu+k)}{2}} (\mathfrak{MF} ^{-1}_{\mu+k} \Phi)(t)dt\\
 	&&=(x^2-1)^{\frac{\mu+k-1}{2}}\frac{1}{\pi}\int^\infty_0 \int_\infty^x	(t^2-1)^{-\frac{(\mu+k)}{2}}P_{-\frac{1}{2}+i\tau}^{-(\mu+k)}(t)\Gamma\Big(\mu+k+\frac{1}{2}+i\tau\Big)\\ 
 	&&\hspace{0.5cm}\times ~\Gamma\Big(\mu+k+\frac{1}{2}-i\tau\Big)~\tau \sinh(\pi\tau)\Phi(\tau)
  dt ~d\tau.
 \end{eqnarray*}
In view of the recurrence relation \eqref{2.3}, we have
\begin{align*}
	D_t [(t^2-1)^{-\frac{(\mu+k-1)}{2}}P_{-\frac{1}{2}+i\tau}^{-(\mu+k-1)}(t)]&=(-1)^{\frac{1}{2}}\Big(\mu+k-\frac{1}{2}+i\tau\Big)_{1}\Big(\mu+k-\frac{1}{2}-i\tau\Big)_{1}\\
	&\quad \times (t^2-1)^{-\frac{(\mu+k)}{2}}P_{-\frac{1}{2}+i\tau}^{-(\mu+k)}(t), 
\end{align*}
and hence
	\begin{eqnarray*}
	&&N_{\mu+k-1}^{-1}\big(\mathfrak{MF}_{\mu+k}^{-1} \Phi\big)(x)\\
	&&=(-1)^{-\frac{1}{2}}(x^2-1)^{\frac{\mu+k-1}{2}}\frac{1}{\pi}\int_0^\infty \int_\infty^x D_t [(t^2-1)^{-\frac{(\mu+k-1)}{2}}P_{-\frac{1}{2}+i\tau}^{-(\mu+k-1)}(t)]\\
	&&\hspace{0.5cm}\times \Gamma\Big(\mu+k-\frac{1}{2}+i\tau\Big) \Gamma\Big(\mu+k-\frac{1}{2}-i\tau\Big)~\tau \sinh(\pi\tau)\Phi(\tau)
	dt ~d\tau.\\
	&&=(-1)^{-\frac{1}{2}}(x^2-1)^{\frac{\mu+k-1}{2}}\frac{1}{\pi}\int_0^\infty(x^2-1)^{-\frac{(\mu+k-1)}{2}}P_{-\frac{1}{2}+i\tau}^{-(\mu+k-1)}(x)\\
	&&\hspace{0.5cm}\times \Gamma\Big(\mu+k-\frac{1}{2}+i\tau\Big) \Gamma\Big(\mu+k-\frac{1}{2}-i\tau\Big)~\tau \sinh(\pi\tau)\Phi(\tau)
	 ~d\tau\\
	&&=(-1)^{-\frac{1}{2}}\big(\mathfrak{MF}_{\mu+k-1}^{-1} \Phi\big)(x).
\end{eqnarray*}
Again 
	\begin{eqnarray*}
	&&N_{\mu+k-2}^{-1}N_{\mu+k-1}^{-1}\big(\mathfrak{MF}_{\mu+k}^{-1} \Phi\big)(x)\\
	&&=(x^2-1)^{\frac{\mu+k-2}{2}}\int_\infty^x (t^2-1)^{-\frac{(\mu+k-1)}{2}} N_{\mu+k-1}^{-1}(\mathfrak{MF} ^{-1}_{\mu+k} \Phi)(t)dt\\
	&&=(-1)^{-\frac{1}{2}}(x^2-1)^{\frac{\mu+k-2}{2}}\frac{1}{\pi}\int^\infty_0 \int_\infty^x	(t^2-1)^{-\frac{(\mu+k-1)}{2}}P_{-\frac{1}{2}+i\tau}^{-(\mu+k-1)}(t)\\ 
	&&\hspace{0.5cm}\times ~\Gamma\Big(\mu+k-\frac{1}{2}+i\tau\Big)\Gamma\Big(\mu+k-\frac{1}{2}-i\tau\Big) ~\tau \sinh(\pi\tau)\Phi(\tau)
	dt ~d\tau.
\end{eqnarray*}
Using the recurrence relation \eqref{2.3}, we have
\begin{align*}
	D_t [(t^2-1)^{-\frac{(\mu+k-2)}{2}}P_{-\frac{1}{2}+i\tau}^{-(\mu+k-2)}(t)]&=(-1)^{\frac{1}{2}}\Big(\mu+k-\frac{3}{2}+i\tau\Big)_{1}\Big(\mu+k-\frac{3}{2}-i\tau\Big)_{1}\\
	&\quad \times (t^2-1)^{-\frac{(\mu+k-1)}{2}}P_{-\frac{1}{2}+i\tau}^{-(\mu+k-1)}(t), 
\end{align*}
and hence
\begin{align*}
	N_{\mu+k-2}^{-1}N_{\mu+k-1}^{-1}\big(\mathfrak{MF}_{\mu+k}^{-1} \Phi\big)(x)
    &=(-1)^{-\frac{2}{2}}\frac{1}{\pi}\int_0^\infty P_{-\frac{1}{2}+i\tau}^{-(\mu+k-2)}(x)\Gamma\Big(\mu+k-\frac{3}{2}+i\tau\Big)\\
    &\quad \times  ~\Gamma\Big(\mu+k-\frac{3}{2}-i\tau\Big)\tau \sinh(\pi\tau)\Phi(\tau)~d\tau\\
&=(-1)^{-\frac{2}{2}}\big(\mathfrak{MF}_{\mu+k-2}^{-1} \Phi\big)(x).
\end{align*}
Proceeding similarly, we can show that
	\begin{eqnarray*}
	N_\mu^{-1}\cdots N_{\mu+k-2}^{-1}N_{\mu+k-1}^{-1}\big(\mathfrak{MF}_{\mu+k}^{-1} \Phi\big)(x)=(-1)^{-\frac{k}{2}}\big(\mathfrak{MF}_{\mu}^{-1} \Phi\big)(x).
\end{eqnarray*}
\end{proof}
\begin{theorem}
	Let $\mu$ be a complex number and $k$ a positive integer such that $\mathfrak{R}(\mu+k)>{-\frac{1}{2}}$. On $\mathbb{H}_\mu(\mathbb{I})$, the inverse $\mathfrak{MF}_{\mu,k}^{-1}$ of the extended Mehler-Fock transform  coincides with $\mathfrak{MF}_{\mu}^{-1}$ defined by \eqref{1.2} whenever $\mathfrak{R}(\mu)>{-\frac{1}{2}}$.
\begin{proof}
	By definition of $\mathfrak{MF}_{\mu,k}^{-1}$ and Lemma 3.3, we get the desired result 
\begin{eqnarray*}
	(\mathfrak{MF}_{\mu,k}^{-1}\Phi)(x)=(-1)^{\frac{k}{2}}N_\mu^{-1}\cdots N_{\mu+k-2}^{-1}N_{\mu+k-1}^{-1}\big(\mathfrak{MF}_{\mu+k}^{-1} \Phi\big)(x)=\big(\mathfrak{MF}_{\mu}^{-1} \Phi\big)(x).
\end{eqnarray*}
\end{proof}
\end{theorem}
A Mehler-Fock transformation of order $\mu$ specified on some generalized functions can be obtained by accessing the dual spaces. Let $f\in\mathbb{H}_\mu^{\prime}(\mathbb{I})$, where $\mu$ is arbitrary. We define the generalized Mehler-Fock transform
$\mathfrak{MF}_\mu^{\prime}$ on $\mathbb{H}_\mu^{\prime}(\mathbb{I})$ as adjoint of $\mathfrak{MF}_\mu^{-1}=\mathfrak{MF}_{\mu,k}^{-1}$ on $\mathbb{H}_\mu(\mathbb{I})$, where $k$ is a positive integer greater than $-\mathfrak{R}(\mu)-\frac{1}{2}$.
That is, for $\Phi\in \mathbb{H}_\mu(\mathbb{I})$ and $\phi=\mathfrak{MF}^{-1}_\mu\Phi=\mathfrak{MF}^{-1}_{\mu,k}\Phi$, the Mehler-Fock transform $F=\mathfrak{MF}_\mu^{\prime}f$ of any $f\in\mathbb{H}_\mu^{\prime}(\mathbb{I})$ is defined as a functional on $\mathbb{H}_\mu(\mathbb{I})$ by \eqref{2.6}; namely
\begin{eqnarray}
	\langle \mathfrak{MF}_\mu^{\prime}f,\Phi \rangle_{(\mathbb{R}_+;\frac{1}{\pi}\tau \sinh(\pi\tau)\Psi(\tau,\mu)d\tau)}=\langle f,\mathfrak{MF}^{-1}_{\mu,k}\Phi \rangle_{(\mathbb{I}; dx)}.\label{3.3}
\end{eqnarray}
In view of \eqref{3.3}, the generalized inverse Mehler-Fock transform $(\mathfrak{MF}_\mu^{\prime})^{-1}$ for any complex $\mu$ is defined by
\begin{eqnarray}
	\langle (\mathfrak{MF}_\mu^{\prime})^{-1}F,\phi \rangle_{(\mathbb{I}; dx)}=\langle F,\Phi \rangle_{(\mathbb{R}_+;\frac{1}{\pi}\tau \sinh(\pi\tau)\Psi(\tau,\mu)d\tau)}.\label{3.4}
\end{eqnarray}
\section{Some operational calculus }
In this section, we establish certain operational properties associated with the differential operator $M_\mu N_\mu$. These results describe the action of the operator in the framework of the arbitrary order Mehler-Fock transform and will be useful in the subsequent analysis.
\begin{theorem}
	Let $\mu$ be a fixed complex number and $k$ a positive integer such that $\mathfrak{R}(\mu+k)>{-\frac{1}{2}}$. Then for $\Phi\in\mathbb{H}_\mu(\mathbb{I})$, 
	\begin{eqnarray}
	M_\mu N_\mu (\mathfrak{MF}_{\mu,k}^{-1}\Phi)(x)=(-1)\bigg[\mathfrak{MF}_{\mu,k}^{-1}\bigg(\Big(\mu+\frac{1}{2}\Big)^2+\tau^2\bigg)\Phi\bigg](x). \label{4.1}
	\end{eqnarray}
	\begin{proof}
		Using \eqref{2.11}, \eqref{3.2} and  Lemma 3.3, we have
			\begin{eqnarray}	
				&&M_\mu N_\mu(\mathfrak{MF}_{\mu,k}^{-1}\Phi)(x)\nonumber\\
			&&=(x^2-1)^{-\frac{\mu}{2}}D_x(x^2-1)^{\mu+1}D_x(x^2-1)^{-\frac{\mu}{2}}\nonumber
			\\
			&&\hspace{0.5cm}\times(-1)^{\frac{k}{2}}N_\mu^{-1}\cdots N_{\mu+k-2}^{-1}N_{\mu+k-1}^{-1}\big(\mathfrak{MF}_{\mu+k}^{-1} \Phi\big)(x)\nonumber\\
			&&=(x^2-1)^{-\frac{\mu}{2}}D_x(x^2-1)^{\mu+1}D_x(x^2-1)^{-\frac{\mu}{2}}\nonumber
			\\
			&&\hspace{0.5cm}\times \frac{1}{\pi}\int\limits_0^\infty \tau \sinh(\pi\tau)\Gamma\Big(\mu+\frac{1}{2}+i\tau\Big)\Gamma\Big(\mu+\frac{1}{2}-i\tau\Big)P_{-\frac{1}{2}+i\tau}^{-\mu}(x)\Phi(\tau)d\tau\nonumber.
		\end{eqnarray}
	Now using \eqref{2.3} and \eqref{2.4}, we have
	\begin{eqnarray}
	M_\mu N_\mu(\mathfrak{MF}_{\mu,k}^{-1}\Phi)(x)\nonumber&=&\frac{(-1)}{\pi}\int_0^\infty \tau \sinh(\pi\tau)\Gamma\Big(\mu+\frac{1}{2}+i\tau\Big)\Gamma\Big(\mu+\frac{1}{2}-i\tau\Big)\nonumber
	\\
	&&\hspace{0.5cm}\times P_{-\frac{1}{2}+i\tau}^{-\mu}(x)\bigg(\Big(\mu+\frac{1}{2}\Big)^2+\tau^2\bigg)\Phi(\tau)d\tau.\label{4.2}
\end{eqnarray}
Now consider the right-hand side of \eqref{4.1}. Applying the definition of $\mathfrak{MF}_{\mu,k}^{-1}$ and using the following relation
	\begin{eqnarray*}
	&&N_\mu^{-1}\cdots N_{\mu+k-2}^{-1}N_{\mu+k-1}^{-1}\bigg[\mathfrak{MF}_{\mu+k}^{-1} \bigg(\Big(\mu+\frac{1}{2}\Big)^2+\tau^2\bigg)\Phi\bigg](x)\nonumber\\
	&&=(-1)^{-\frac{k}{2}}\frac{1}{\pi}\int_0^\infty \tau \sinh(\pi\tau)\Gamma\Big(\mu+\frac{1}{2}+i\tau\Big)\Gamma\Big(\mu+\frac{1}{2}-i\tau\Big)\nonumber
	\\
	&&\hspace{0.5cm}\times P_{-\frac{1}{2}+i\tau}^{-\mu}(x)\bigg(\Big(\mu+\frac{1}{2}\Big)^2+\tau^2\bigg)\Phi(\tau)d\tau,
\end{eqnarray*}
we have
\begin{eqnarray}
	&&(-1)\bigg[\mathfrak{MF}_{\mu,k}^{-1}\bigg(\Big(\mu+\frac{1}{2}\Big)^2+\tau^2\bigg)\Phi\bigg](x)\nonumber\\
	&&=\frac{(-1)}{\pi}\int_0^\infty \tau \sinh(\pi\tau)\Gamma\Big(\mu+\frac{1}{2}+i\tau\Big)\Gamma\Big(\mu+\frac{1}{2}-i\tau\Big)\nonumber
	\\
	&&\times P_{-\frac{1}{2}+i\tau}^{-\mu}(x)\bigg(\Big(\mu+\frac{1}{2}\Big)^2+\tau^2\bigg)\Phi(\tau)d\tau.\label{4.3}
\end{eqnarray}
Therefore, from \eqref{4.2} and \eqref{4.3}, we get the desired result \eqref{4.1}.
	\end{proof}
\end{theorem}
\textbf{Note:} For $r\in \mathbb{N}_{0}$, we have 
$$ (M_\mu N_\mu)^r(\mathfrak{MF}_{\mu,k}^{-1}\Phi)(x)=(-1)^r\bigg[\mathfrak{MF}_{\mu,k}^{-1}\bigg(\Big(\mu+\frac{1}{2}\Big)^2+\tau^2\bigg)^r\Phi\bigg](x).
$$
\begin{theorem}
	Let $\mu$ be a fixed complex number and $k$ be any positive integer such that $\mathfrak{R}(\mu+k)>{-\frac{1}{2}}$. Then for $\phi\in\mathbb{H}_\mu(\mathbb{I})$, 
	\begin{eqnarray}
\big(\mathfrak{MF}_{\mu+1,k}N_\mu\phi\big)(\tau)=(-1)^{-\frac{1}{2}}(\mathfrak{MF}_{\mu,k}\phi)(\tau).\label{4.4}
	\end{eqnarray}	
	\begin{proof}Let $\phi\in\mathbb{H}_\mu(\mathbb{I})$. Then using \eqref{3.1} and Corollary 3.2, we have
		\begin{eqnarray*}
			\big(\mathfrak{MF}_{\mu+1,k}N_\mu\phi\big)(\tau)&=&(-1)^{\frac{k}{2}}\big(\mathfrak{MF}_{\mu+1+k}N_{\mu+k}\cdots N_{\mu+1}(N_\mu \phi)\big)(\tau)\\
			&=&(-1)^{-\frac{1}{2}}\big(\mathfrak{MF}_{\mu,k+1}\phi\big)(\tau)\\
			&=&(-1)^{-\frac{1}{2}}\big(\mathfrak{MF}_{\mu,k}\phi\big)(\tau).
		\end{eqnarray*}	
	\end{proof}
\end{theorem}
\begin{theorem}
	Let $\mu$ be a fixed complex number and $k$ a positive integer such that $\mathfrak{R}(\mu+k)>{-\frac{1}{2}}$. Then\\\\
	(i) for $\phi \in\mathbb{H}_{\mu+1}(\mathbb{I})$, we have
\begin{eqnarray}
 \big(\mathfrak{MF}_{\mu,k}M_\mu\phi\big)(\tau)
	=(-1)^{-\frac{1}{2}}\bigg[\Big(\mu+\frac{1}{2}\Big)^2+\tau^2\bigg]\big(\mathfrak{MF}_{\mu+1,k}\phi\big)(\tau).  \label{4.5}
\end{eqnarray}
	(ii)~ For $\phi \in\mathbb{H}_{\mu}(\mathbb{I})$, we have
	\begin{eqnarray}
		(\mathfrak{MF}_{\mu,k}A_\mu\phi)(\tau)=(-1)\bigg[\Big(\mu+\frac{1}{2}\Big)^2+\tau^2\bigg](\mathfrak{MF}_{\mu,k}\phi)(\tau). \label{4.6}
	\end{eqnarray}	 	
	\begin{proof}(i) From \cite[Lemma 2.5(ii)]{AP2}, we have 
			\begin{eqnarray}
			N_{\mu+k-1}\cdots N_{\mu+1}N_\mu{\phi(x)}=(x^2-1)^{\frac{\mu+k}{2}}D_x^k(x^2-1)^{-\frac{\mu}{2}}\phi(x).\label{4.7}
		\end{eqnarray}
	Now using \eqref{3.1}, \eqref{2.7} and \eqref{4.7}, we get
		\begin{eqnarray*}
			\big(\mathfrak{MF}_{\mu,k}M_\mu\phi\big)(\tau)&=&(-1)^{\frac{k}{2}}\int_1^\infty P_{-\frac{1}{2}+i\tau}^{-(\mu+k)}(x)(x^2-1)^\frac{\mu+k}{2} D_x^k(x^2-1)^{-\mu}\\
			&&\times D_x(x^2-1)^\frac{\mu+1}{2}\phi(x)dx.
		\end{eqnarray*}	
	Integrating by parts $k$ times and using \eqref{2.4}, we have
		\begin{eqnarray*}
		\big(\mathfrak{MF}_{\mu,k}M_\mu\phi\big)(\tau)&=&(-1)^{\frac{3k}{2}}\int_1^\infty D_x^k\big[(x^2-1)^\frac{\mu+k}{2}P_{-\frac{1}{2}+i\tau}^{-(\mu+k)}(x)\big](x^2-1)^{-\mu}\\
		&&\times D_x(x^2-1)^\frac{\mu+1}{2}\phi(x)dx\\
		&=&\int_1^\infty P_{-\frac{1}{2}+i\tau}^{-\mu}(x)(x^2-1)^{-\frac{\mu}{2}} D_x(x^2-1)^\frac{\mu+1}{2}\phi(x)dx
		\end{eqnarray*}	
Again Integrating by parts and using \eqref{2.3}, we have
\begin{eqnarray}
\big(\mathfrak{MF}_{\mu,k}M_\mu\phi\big)(\tau)=(-1)^{-\frac{1}{2}}\bigg[\Big(\mu+\frac{1}{2}\Big)^2+\tau^2\bigg]\int_1^\infty P_{-\frac{1}{2}+i\tau}^{-(\mu+1)}(x)\phi(x)dx.\label{4.8}
\end{eqnarray}	
We now consider the right side of \eqref{4.5}. Applying the definition of $\mathfrak{MF}_{\mu+1,k}$ and using the following relation 
\begin{eqnarray*}
	N_{\mu+k}N_{\mu+k-1}\cdots N_{\mu+1}\phi(x)=(x^2-1)^{\frac{\mu+k+1}{2}}D_x^k(x^2-1)^{-\frac{(\mu+1)}{2}}\phi(x),
\end{eqnarray*}
we have
\begin{eqnarray*}
\big(\mathfrak{MF}_{\mu+1,k}\phi\big)(\tau)
&=&\big(\mathfrak{MF}_{\mu+1+k}N_{\mu+k}N_{\mu+k-1}\cdots N_{\mu+1} \phi\big)(\tau)\\
&=&\int_1^\infty P_{-\frac{1}{2}+i\tau}^{-(\mu+k+1)}(x)(x^2-1)^\frac{\mu+k+1}{2}
D_x^k(x^2-1)^{-\frac{(\mu+1)}{2}}\phi(x)dx.
\end{eqnarray*}
Now integrating by parts $k$ times and using \eqref{2.4}, we get
\begin{eqnarray}
	\big(\mathfrak{MF}_{\mu+1,k}\phi\big)(\tau)\nonumber
	&=&\int_1^\infty D_x^k \big[(x^2-1)^\frac{\mu+k+1}{2}P_{-\frac{1}{2}+i\tau}^{-(\mu+k+1)}(x)\big]\nonumber (x^2-1)^{-\frac{(\mu+1)}{2}}\phi(x)dx \nonumber\\
	&=&\int_1^\infty P_{-\frac{1}{2}+i\tau}^{-(\mu+1)}(x)\phi(x)dx.\label{4.9}
\end{eqnarray}
Therefore, from \eqref{4.8} and \eqref{4.9}, we get the desired result \eqref{4.5}.\\\\
(ii) By combining \eqref{4.4} and \eqref{4.5}, we get \eqref{4.6}.
	\end{proof}
\end{theorem}
The analogue to Theorem 4.1 in the dual space is 
\begin{theorem}
	For $f\in \mathbb{H}_{\mu}^{\prime}(\mathbb{I})$, we have 
		\begin{eqnarray}
		(\mathfrak{MF}_{\mu}^{\prime}M_\mu N_\mu f)(\tau)=-\bigg[\Big(\mu+\frac{1}{2}\Big)^2+\tau^2\bigg](\mathfrak{MF}_{\mu}^{\prime}f)(\tau). \label{4.10}
	\end{eqnarray}
\begin{proof}
	Let $\Phi\in\mathbb{H}_{\mu}(\mathbb{I})$, and $k$ be any positive integer such that $\mathfrak{R}(\mu+k)>{-\frac{1}{2}}$. Then, by the definition of $\mathfrak{MF}_{\mu}^{\prime}$, and $M_\mu N_\mu$ when acting on $\mathbb{H}_{\mu}^{\prime}(\mathbb{I})$, and using Theorem 4.1, we have 	
		\begin{eqnarray*}
		&&\langle \mathfrak{MF}_{\mu}^{\prime}M_\mu N_\mu f,    \Phi\rangle_{(\mathbb{R}_+;\frac{1}{\pi}\tau \sinh(\pi\tau)\Psi(\tau,\mu)d\tau)}\\
		&&=\langle M_\mu N_\mu f,  \mathfrak{MF}_{\mu,k}^{-1}\Phi\rangle_{(\mathbb{I}; dx)}\\
		&&=\langle f,  M_\mu N_\mu \mathfrak{MF}_{\mu,k}^{-1}\Phi\rangle_{(\mathbb{I}; dx)}\\
		&&=\bigg\langle f, \mathfrak{MF}_{\mu,k}^{-1}\bigg(-\Big(\mu+\frac{1}{2}\Big)^2-\tau^2\bigg)\Phi\bigg\rangle_{\big(\mathbb{I}; dx\big)}\\
		&&=\bigg\langle \mathfrak{MF}_{\mu}^{\prime} f, \bigg(-\Big(\mu+\frac{1}{2}\Big)^2-\tau^2\bigg)\Phi\bigg\rangle_{\big(\mathbb{R}_+; \frac{1}{\pi}\tau \sinh(\pi\tau)\Psi(\tau,\mu)d\tau\big)}.
	\end{eqnarray*}
But, multiplication by $-\bigg[\Big(\mu+\frac{1}{2}\Big)^2+\tau^2\bigg]$ in $\mathbb{H}_{\mu}^{\prime}(\mathbb{I})$ is the adjoint of  multiplication by $-\bigg[\Big(\mu+\frac{1}{2}\Big)^2+\tau^2\bigg]$ in $\mathbb{H}_{\mu}(\mathbb{I})$. Thus the last expression is equal to 
\begin{eqnarray*}
\bigg\langle -\bigg[\Big(\mu+\frac{1}{2}\Big)^2+\tau^2\bigg]\mathfrak{MF}_{\mu}^{\prime} f, \Phi\bigg\rangle_{\big(\mathbb{R}_+;\frac{1}{\pi}\tau \sinh(\pi\tau)\Psi(\tau,\mu)d\tau\big)}.
\end{eqnarray*}
This establishes \eqref{4.10}.
\end{proof}
\end{theorem}

\section{The Mehler-Fock convolution of arbitrary order}
Some important contributions to the study of convolution for the MFT have been made by \cite{AP2,per,Ha1,Glaeske}. However, the MFT of arbitrary order as well as convolution associated with the MFT of arbitrary order have still not been explored in any functional spaces. In this section, we define the translation and convolution operators associated with the MFT of arbitrary order. Moreover, we demonstrate that the factorization property for this convolution holds in the space $\mathbb{H}_\mu$. Let $\mu\in \mathbb{C}$ and $k$ be any positive integer such that $\mathfrak{R}(\mu+k)>{-\frac{1}{2}}$. Then for $f\in \mathbb{H}_\mu(\mathbb{I})$, the translation operator $(T_y f)(x)$ for the MFT of arbitrary order is defined by 
\begin{align}
(T_x f)(y)=(T_y f)(x)&=(-1)^{\frac{k}{2}}\frac{1}{\pi}\int_0^\infty \int_1^\infty P_{-\frac{1}{2}+i\tau}^{-(\mu+k)}(z)(z^2-1)^\frac{\mu+k}{2} D_z^k f(z) \nonumber\\ 
&\times\tau\sinh(\pi\tau)\Psi(\tau, \mu)P_{-\frac{1}{2}+i\tau}^{-\mu}(x)P_{-\frac{1}{2}+i\tau}^{-\mu}(y) dzd\tau, \label{5.1}
\end{align}
and the convolution operator $h(x)$ for the MFT of arbitrary order of functions $f\in \mathbb{H}_\mu(\mathbb{I}) $ and $ g\in \mathbb{H}_\mu(\mathbb{I}) $ is given by
\begin{align}
	h(x)&=(-1)^{k}\frac{1}{\pi}\int_0^\infty \int_1^\infty\int_1^\infty P_{-\frac{1}{2}+i\tau}^{-(\mu+k)}(z)(z^2-1)^\frac{\mu+k}{2}D_z^kf(z)~P_{-\frac{1}{2}+i\tau}^{-(\mu+k)}(y) \nonumber\\ 
	&\times (y^2-1)^\frac{\mu+k}{2}D_y^k g(y) ~\tau \sinh(\pi\tau)\Psi(\tau, \mu)P_{-\frac{1}{2}+i\tau}^{-\mu}(x) dz dy d\tau. \label{5.2}
\end{align}
\begin{theorem}
	Let $\mu$ be a complex number and $k$ a positive integer such that $\mathfrak{R}(\mu+k)>{-\frac{1}{2}}$. Then for $f, g\in\mathbb{H}_\mu(\mathbb{I})$, we have
	\begin{eqnarray}
		(i)~[\mathfrak{MF}_{\mu,k}(T_yf)(x)](\tau)&=&P_{-\frac{1}{2}+i\tau}^{-\mu}(y)~[\mathfrak{MF}_{\mu,k}\big((z^2-1)^{\frac{\mu}{2}}f(z)\big)](\tau),\label{5.3}
	\end{eqnarray}
	\begin{eqnarray}
		(ii)~[\mathfrak{MF}_{\mu,k} h(x)](\tau)=[\mathfrak{MF}_{\mu,k}\big((z^2-1)^{\frac{\mu}{2}}f(z)\big)](\tau)~
		[\mathfrak{MF}_{\mu,k}\big((y^2-1)^{\frac{\mu}{2}}g(y)\big)](\tau).\label{5.4}
	\end{eqnarray}
	\begin{proof}
		Clearly 
		\begin{align}
			(T_y f)(x)=(-1)^{\frac{k}{2}}\bigg[\mathfrak{MF}^{-1}_\mu \Big(\int_1^\infty P_{-\frac{1}{2}+i\tau}^{-(\mu+k)}(z)(z^2-1)^\frac{\mu+k}{2} D_z^k f(z)P_{-\frac{1}{2}+i\tau}^{-\mu}(y) dz\Big)\bigg](x). \label{5.5}
		\end{align}
	According to \cite[Theorem 2.8]{AP2}, the following result holds
			\begin{eqnarray}
			\big(\mathfrak{MF}_{\mu+k}N_{\mu+k-1}\cdots N_\mu(T_yf)(x)\big)(\tau)=(-1)^{-\frac{k}{2}}\big(\mathfrak{MF}_{\mu}(T_yf)(x)\big)(\tau).\label{5.6}
		\end{eqnarray}	
	By applying \eqref{5.5} to the right-hand side of \eqref{5.6}, we obtain
		\begin{eqnarray*}
		&&\big(\mathfrak{MF}_{\mu+k}N_{\mu+k-1}\cdots N_\mu(T_yf)(x)\big)(\tau)\\
		&&=P_{-\frac{1}{2}+i\tau}^{-\mu}(y)\int_1^\infty P_{-\frac{1}{2}+i\tau}^{-(\mu+k)}(z)(z^2-1)^\frac{\mu+k}{2} D_z^k f(z)dz\\
		&&=P_{-\frac{1}{2}+i\tau}^{-\mu}(y) \big[\mathfrak{MF}_{\mu+k}\big((z^2-1)^\frac{\mu+k}{2}D_z^k f(z)\big)\big](\tau)\\
		&&=P_{-\frac{1}{2}+i\tau}^{-\mu}(y) \big[\mathfrak{MF}_{\mu+k}N_{\mu+k-1}\cdots N_\mu\big((z^2-1)^\frac{\mu}{2}f(z)\big)\big](\tau)\\
		&&=(-1)^{-\frac{k}{2}}P_{-\frac{1}{2}+i\tau}^{-\mu}(y)[\mathfrak{MF}_{\mu,k}\big((z^2-1)^{\frac{\mu}{2}}f(z)\big)](\tau).
		\end{eqnarray*}	
	Hence, we get 
	\begin{eqnarray*}
	[\mathfrak{MF}_{\mu,k}(T_yf)(x)](\tau)&=&P_{-\frac{1}{2}+i\tau}^{-\mu}(y)~[\mathfrak{MF}_{\mu,k}\big((z^2-1)^{\frac{\mu}{2}}f(z)\big)](\tau).
		\end{eqnarray*}	
	
	(ii) Clearly 
	\begin{eqnarray}
		h(x)&=&(-1)^k\bigg[\mathfrak{MF}^{-1}_\mu \Big(\int_1^\infty\int_1^\infty P_{-\frac{1}{2}+i\tau}^{-(\mu+k)}(z)(z^2-1)^\frac{\mu+k}{2}D_z^kf(z)\nonumber\\
		 &&\times P_{-\frac{1}{2}+i\tau}^{-(\mu+k)}(y)  (y^2-1)^\frac{\mu+k}{2}D_y^k g(y) dzdy\Big)\bigg](x). \label{5.7}
	\end{eqnarray}
According to \cite[Theorem 2.8]{AP2}, the following result holds
	\begin{eqnarray}
		\big(\mathfrak{MF}_{\mu+k}N_{\mu+k-1}\cdots N_\mu h(x)\big)(\tau)=(-1)^{-\frac{k}{2}}\big(\mathfrak{MF}_{\mu}h(x)\big)(\tau).\label{5.8}
	\end{eqnarray}	
Now using \eqref{5.7} to the right-hand side of \eqref{5.8}, we have
\begin{eqnarray*}
	&&\big(\mathfrak{MF}_{\mu+k}N_{\mu+k-1}\cdots N_\mu h(x)\big)(\tau)\\
	&&=(-1)^{\frac{k}{2}}\int_1^\infty\int_1^\infty P_{-\frac{1}{2}+i\tau}^{-(\mu+k)}(z)(z^2-1)^\frac{\mu+k}{2}D_z^kf(z)\\
	&&\hspace{0.5cm}\times P_{-\frac{1}{2}+i\tau}^{-(\mu+k)}(y)  (y^2-1)^\frac{\mu+k}{2}D_y^k g(y) dzdy\\
	&&=(-1)^{\frac{k}{2}} \big[\mathfrak{MF}_{\mu+k}\big((z^2-1)^\frac{\mu+k}{2}D_z^k f(z)\big)\big](\tau)~\big[\mathfrak{MF}_{\mu+k}\big((y^2-1)^\frac{\mu+k}{2}D_y^k g(y)\big)\big](\tau)\\
	&&=(-1)^{\frac{k}{2}} \big[\mathfrak{MF}_{\mu+k}N_{\mu+k-1}\cdots N_\mu\big((z^2-1)^\frac{\mu}{2}f(z)\big)\big](\tau)\\
	&&\hspace{0.5cm}\times \big[\mathfrak{MF}_{\mu+k}N_{\mu+k-1}\cdots N_\mu\big((y^2-1)^\frac{\mu}{2}g(y)\big)\big](\tau)\\
	&&=(-1)^{-\frac{k}{2}}[\mathfrak{MF}_{\mu,k}\big((z^2-1)^{\frac{\mu}{2}}f(z)\big)](\tau)~[\mathfrak{MF}_{\mu,k}\big((y^2-1)^{\frac{\mu}{2}}g(y)\big)](\tau).
\end{eqnarray*}	
Hence, we get 
\begin{eqnarray*}
[\mathfrak{MF}_{\mu,k} h(x)](\tau)=[\mathfrak{MF}_{\mu,k}\big((z^2-1)^{\frac{\mu}{2}}f(z)\big)](\tau)~
[\mathfrak{MF}_{\mu,k}\big((y^2-1)^{\frac{\mu}{2}}g(y)\big)](\tau).
\end{eqnarray*}
		\end{proof}
\end{theorem}
\section{Applications}
\begin{problem} With the help of generalized Mehler-Fock transform $\mathfrak{MF}_{\mu}^{\prime}$ and Theorem 4.4, one can solve the differential equation of the type
\begin{eqnarray}
	P(M_\mu N_\mu)v=g,\label{6.1}
\end{eqnarray}
where $P(\cdot)$ be any polynomial having no zeroes in the interval $\bigg(-\infty,-\Big(\mu+\frac{1}{2}\Big)^2 \bigg)$, $g\in \mathbb{H}_{\mu}^{\prime}(\mathbb{I})$, $v$ is unknown but required to be in $\mathbb{H}_{\mu}^{\prime}(\mathbb{I})$ and $\mu$ be any complex number. 
	\end{problem}
 Applying the generalized Mehler-Fock transform $\mathfrak{MF}_{\mu}^{\prime}$ to \eqref{6.1} and using Theorem 4.4, we may write
\begin{eqnarray}
	P\bigg(-\Big(\mu+\frac{1}{2}\Big)^2-\tau^2\bigg)V(\tau)=G(\tau),\label{6.2}
\end{eqnarray}
where, $V(\tau)=\mathfrak{MF}_{\mu}^{\prime}v$, and $G(\tau)=\mathfrak{MF}_{\mu}^{\prime}g.$ Therefore if $	P\bigg(-\Big(\mu+\frac{1}{2}\Big)^2-\tau^2\bigg)\neq 0$, we can divide both sides of \eqref{6.2} by $P\bigg(-\Big(\mu+\frac{1}{2}\Big)^2-\tau^2\bigg)$ and subsequently apply the generalized inverse Mehler-Fock transform $(\mathfrak{MF}_{\mu}^{\prime})^{-1}$ to obtain the solution
 \begin{eqnarray}
 v=(\mathfrak{MF}_{\mu}^{\prime})^{-1}\frac{G(\tau)}{	P\bigg(-\Big(\mu+\frac{1}{2}\Big)^2-\tau^2\bigg)}.\label{6.3}
 \end{eqnarray}
This means that for each member $\phi$ of $\mathbb{H}_{\mu}(\mathbb{I})$ and $\Phi=\mathfrak{MF}_{\mu,k}\phi$, the unknown $v\in\mathbb{H}_{\mu}^{\prime}(\mathbb{I})$ is given (in view of \eqref{3.4}) by
 \begin{eqnarray*}
	\langle v,\phi\rangle_{(\mathbb{I}; dx)} &=&\bigg\langle (\mathfrak{MF}_{\mu}^{\prime})^{-1}\frac{G(\tau)}{	P\Big(-\Big(\mu+\frac{1}{2}\Big)^2-\tau^2\Big)},\phi\bigg\rangle_{\big(\mathbb{I}; dx\big)}
    \end{eqnarray*}
    \begin{eqnarray*}
	&=&\bigg\langle G(\tau),\frac{\Phi(\tau)}{P\Big(-\Big(\mu+\frac{1}{2}\Big)^2-\tau^2\Big)}\bigg\rangle_{\big(\mathbb{R}_+;\frac{1}{\pi}\tau \sinh(\pi\tau)\Psi(\tau,\mu)d\tau\big)}\\
   &=&\bigg\langle g(x),\mathfrak{MF}^{-1}_{\mu,k}\frac{\Phi(\tau)}{P\Big(-\Big(\mu+\frac{1}{2}\Big)^2-\tau^2\Big)}\bigg\rangle_{\big(\mathbb{I}; dx\big)}\\
   &=&\bigg\langle g(x), (-1)^{\frac{k}{2}}N_\mu^{-1}\cdots N_{\mu+k-1}^{-1}\mathfrak{MF}_{\mu+k}^{-1} \frac{\Phi(\tau)}{P\Big(-\Big(\mu+\frac{1}{2}\Big)^2-\tau^2\Big)}\bigg\rangle_{\big(\mathbb{I}; dx\big)}.
	\end{eqnarray*}

\begin{problem}  Let $\mu\in \mathbb{C}$ and $k\in \mathbb{N}$ such that $\mathfrak{R}(\mu+k)>{-\frac{1}{2}}$. Then for $l>1$, we have 
\begin{eqnarray}
	\big(\mathfrak{MF}_{\mu,k}\delta(x-l)\big)(\tau)=P_{-\frac{1}{2}+i\tau}^{-\mu}(l) .\label{6.4}
\end{eqnarray}
\end{problem}
\begin{proof}
	Using \eqref{3.1}, we get
	\begin{eqnarray*}
		&&\big(\mathfrak{MF}_{\mu,k}\delta(x-l)\big)(\tau)\\
		&&=(-1)^{\frac{k}{2}}\big(\mathfrak{MF}_{\mu+k}N_{\mu+k-1}\cdots N_\mu \delta(x-l)\big)(\tau)\\
		&&=(-1)^{\frac{k}{2}}\int_1^\infty P_{-\frac{1}{2}+i\tau}^{-(\mu+k)}(x)(x^2-1)^{\frac{\mu+k}{2}}D_x^k(x^2-1)^{-\frac{\mu}{2}}\delta(x-l)dx.
	\end{eqnarray*}
Integrating by parts and using \eqref{2.4}, we get
	\begin{eqnarray*}
	&&\big(\mathfrak{MF}_{\mu,k}\delta(x-l)\big)(\tau)\\
	&&=(-1)^{\frac{k}{2}}(-1)^{1+{\frac{1}{2}}}\int_1^\infty P_{-\frac{1}{2}+i\tau}^{-(\mu+k-1)}(x)(x^2-1)^{\frac{\mu+k-1}{2}}D_x^{k-1}(x^2-1)^{-\frac{\mu}{2}}\delta(x-l)dx.
\end{eqnarray*}
Repeating this process $k$ times, we get the desired result \eqref{6.4}.
	\end{proof}
\begin{problem} 
Using the Mehler-Fock transform of arbitrary order, we analyze the solution of generalized Heat equation 
	\begin{eqnarray}
		\frac{\partial \phi}{\partial t}=K A_\mu\phi(x,t), \quad x\in(1,\infty),~ t>0, \label{6.6}
	\end{eqnarray}
\begin{eqnarray}
\phi(x,0)=f(x), \quad x\in(1,\infty). \label{6.7}
\end{eqnarray}
\end{problem}
Let us consider $\big(\mathfrak{MF}_{\mu,k}\phi(x,t)\big)(\tau,t)=\Phi(\tau,t)$. Then in view of \eqref{4.6}, \eqref{6.6} becomes
\begin{eqnarray}
	\frac{d}{dt}\Phi(\tau,t)+K \bigg[\Big(\mu+\frac{1}{2}\Big)^2+\tau^2\bigg]\Phi(\tau,t)=0. \label{6.8}
\end{eqnarray}
The solution of \eqref{6.8} is given by 
\begin{eqnarray}
\Phi(\tau,t)=C~ e^{{-K}\big[(\mu+\frac{1}{2})^2+\tau^2\big]t}. \label{6.9}
\end{eqnarray}
Similarly from \eqref{6.7}, we have
\begin{eqnarray}
	\Phi(\tau,0)=F(\tau). \label{6.10}
\end{eqnarray}
Setting $t=0$ in \eqref{6.9} and using \eqref{6.10}, we have $C=F(\tau).$ Hence \eqref{6.9} becomes
\begin{eqnarray*}
	\Phi(\tau,t)=F(\tau) e^{{-K}\big[(\mu+\frac{1}{2})^2+\tau^2\big]t}.
\end{eqnarray*}
Taking inverse Mehler-Fock transform of arbitrary order and using Lemma 3.3, we get the formal solution 
 \begin{eqnarray*}
 	\phi(x,t)=\mathfrak{MF}_{\mu}^{-1}\bigg(F(\tau) e^{{-K}\big[(\mu+\frac{1}{2})^2+\tau^2\big]t}\bigg).
 \end{eqnarray*}
\\
{\bf Data Availability:} The manuscript has no associated data.\\ \\
{\bf Declarations:\\
Conflict of interest:} No ptential conflict of interest was reported by the authors.

\bibliographystyle{amsplain}

\begin{thebibliography}{99}

\bibitem{bry}
Y.A. Brychkov, 
Handbook of Special Functions: Derivatives, Integrals, Series and Other Formulas, 
CRC Press, Boca Raton, 2008.

\bibitem{fock}
V.A. Fock, 
On the representation of an arbitrary function by integrals involving Legendre functions with a complex index, 
Dokl. AN SSSR 39 (7) (1943) 279--283 (in Russian).

\bibitem{per}
H.J. Glaeske, A. Hess, 
On the convolution theorem of the Mehler--Fock transform for a class of generalized functions (I), 
Math. Nachr. 131 (1) (1987) 107--117.

\bibitem{Glaeske}
H.J. Glaeske, A. Hess, 
On the convolution theorem of the Mehler--Fock transform for a class of generalized functions (II), 
Math. Nachr. 136 (1) (1988) 119--129.

\bibitem{gonz1}
B.J. Gonz\'alez, E.R. Negr\'in, 
Abelian theorems for distributional Kontorovich--Lebedev and Mehler--Fock transforms of general order, 
Banach J. Math. Anal. 13 (3) (2019) 524--537.

\bibitem{gonz3}
B.J. Gonz\'alez, E.R. Negr\'in, 
$L_p$-inequalities and Parseval-type relations for the Mehler--Fock transform of general order, 
Ann. Funct. Anal. 8 (2) (2017) 231--239.

\bibitem{Bn1}
B.J. Gonz\'alez, E.R. Negr\'in, 
Mehler--Fock transforms of generalized functions via the method of adjoints, 
Proc. Amer. Math. Soc. 125 (11) (1997) 3243--3253.

\bibitem{Gupt}
A.K. Gupt, U.K. Mandal, A. Prasad, 
Lebedev--Skalskaya transform related continuous wavelet transform, 
Results Math. 79 (3) (2024) 1--21.

\bibitem{Ha1}
N. Hayek, B.J. Gonz\'alez, 
Convolution for the Mehler--Fock transform of generalized functions, 
Integral Transforms Spec. Funct. 9 (4) (2000) 257--270.

\bibitem{Ha2}
N. Hayek, B.J. Gonz\'alez, 
On the Mehler--Fock transform of generalized functions, 
Bull. Soc. Roy. Sci. Li\`ege 61 (3--4) (1992) 315--327.

\bibitem{low}
J.S. Lowndes, 
Note on the generalized Mehler transform, 
Proc. Cambridge Philos. Soc. 60 (1) (1964) 57--59.

\bibitem{Maan}
J. Maan, E.R. Negr\'in, 
Index Whittaker transforms over Lebesgue spaces, 
J. Pseudo-Differ. Oper. Appl. 16 (1) (2025) 1--15.

\bibitem{mehler}
F.G. Mehler, 
Ueber eine mit den kugel-und cylinderfunctionen verwandte function und ihre anwendung in der theorie der elektricitatsvertheilung, 
Math. Ann. 18 (1881) 161--194.


\bibitem{AP1}
A. Prasad, U.K. Mandal, S. Ranjan, 
The $\mu$th order Mehler--Fock transform and allied pseudo-differential operator, 
Integral Transforms Spec. Funct. 36 (3) (2024) 163--175. 

\bibitem{AP2}
A. Prasad, S. Ranjan, 
An operational calculus of $\mu$th order Mehler--Fock transform, 
Complex Anal. Oper. Theory 19 (2) (2025), Article 41, 1--33.

\bibitem{PS}
A. Prasad, S.K. Verma, U.K. Mandal, 
The convolution for zero-order Mehler--Fock transform and pseudo-differential operator, 
Integral Transforms Spec. Funct. 29 (3) (2018) 189--206.

\bibitem{Sn}
I.N. Sneddon, 
The Use of Integral Transforms, 
McGraw--Hill, New York, 1972.

\bibitem{Hm}
H.M. Srivastava, B.J. Gonz\'alez, E.R. Negr\'in, 
New $L^p$-boundedness properties for the Kontorovich--Lebedev and Mehler--Fock transforms, 
Integral Transforms Spec. Funct. 27 (10) (2016) 835--845.

\bibitem{Yaku2}
S. Yakubovich, 
An index integral and convolution operator related to the Kontorovich--Lebedev and Mehler--Fock transforms, 
Complex Anal. Oper. Theory 6 (4) (2012) 947--970.

\bibitem{index}
S.B. Yakubovich, 
Index Transforms, 
World Scientific, Singapore, 1996.

\end{thebibliography}

\end{document}